\documentclass[11pt]{article}
\usepackage[margin=1in]{geometry}
\usepackage{amsmath,amssymb,amsthm,mathtools}
\usepackage{xcolor, comment}
\usepackage[colorlinks=true,linkcolor=blue!55!black,citecolor=green!45!black,urlcolor=blue!55!black]{hyperref}
\usepackage{microtype}
\usepackage{enumitem}

\newtheorem{theorem}{Theorem}[section]
\newtheorem{lemma}[theorem]{Lemma}
\newtheorem{proposition}[theorem]{Proposition}

\title{Analytic discs and compactness of the $\bar\partial$-Neumann operator}

\author{Qianyun Wang, Yuan Yuan, Xu Zhang}
\date{}

\begin{document}
\maketitle

\begin{abstract}
We construct a  bounded pseudoconvex complete Reinhardt domain $\Omega$ with smooth boundary in $\mathbb{C}^3$
such that the \(\bar\partial\)-Neumann operator \(N_1\) is compact
although \(b\Omega\) contains an analytic disc and thus also fails Catlin's Property
\((P)\) and McNeal's Property \((\tilde P)\). 
This example solves an open problem on compactness of the
\(\bar\partial\)-Neumann operator in the negative.
\end{abstract}

\section{Introduction}

Let $\Omega\Subset\mathbb C^n$ be a bounded pseudoconvex domain and let
$1\leq q\leq n$.  The $\bar\partial$-Neumann operator $N_q$ is the inverse,
on the orthogonal complement of its kernel, of the complex Laplacian
$   \square_q=\bar\partial\bar\partial^*   +\bar\partial^*\bar\partial$
acting on $(0,q)$-forms.  The foundational $L^2$ theory yields existence and
boundedness of $N_q$ on bounded pseudoconvex domains. The reader may refer to \cite{FK72,CS01,S10} for
the basic analytic
framework and its subsequent developments.  Because the boundary conditions for $\square_q$ are
not coercive, boundary geometry enters the estimates in an essential way.
On strongly pseudoconvex domains one has a subelliptic estimate, hence a
positive gain of derivatives.  More generally, finite type geometry leads
to subellipticity. 
Subellipticity implies compactness by the Rellich lemma and also yields global
Sobolev regularity.  Global regularity, however, is a weaker issue: it asks
that $N_q$ preserve Sobolev spaces, without requiring a positive gain of
derivatives or discreteness of the spectrum.  Catlin's weighted methods
provided important sufficient conditions for global regularity
\cite{C83}, and the principal approaches and their relations are surveyed in
\cite{BS99,S06}.
The distinction between existence and regularity became especially vivid
with the Diederich--Forn\ae ss worm domain.  Barrett proved that on worm domains the
Bergman projection fails to preserve certain Sobolev spaces \cite{B92}, and
Christ subsequently established global $C^\infty$ irregularity of the
$\bar\partial$-Neumann problem \cite{C96}.  Thus smooth bounded
pseudoconvexity alone does not ensure global regularity.  These examples also
show why estimates weaker than subelliptic estimates, but strong enough to
control the low-energy spectrum, became a central object of study.

The compactness method has its origin in Kohn and Nirenberg's treatment of
non-coercive boundary value problems \cite{KN65}.  In the present setting its
operator-theoretic form is the compactness estimate
\begin{equation}\label{eq:intro-compactness-estimate}
  \|u\|^2
  \leq \varepsilon\bigl(\|\bar\partial u\|^2
                    +\|\bar\partial^*u\|^2\bigr)
       +C_\varepsilon\|u\|_{-1}^2,
  \quad
  u\in\operatorname{Dom}(\bar\partial)
       \cap\operatorname{Dom}(\bar\partial^*),
\end{equation}
for every $\varepsilon>0$.  Equivalently, the inclusion of the form in $\operatorname{Dom}(\bar\partial)
       \cap\operatorname{Dom}(\bar\partial^*)$,
equipped with its graph norm, into $L^2_{(0,q)}(\Omega)$ is compact; equivalently,
$N_q$ is compact, or $\square_q$ has compact resolvent.  In particular,
compactness is a spectral condition rather than merely a regularity
condition.  On a smoothly bounded pseudoconvex domain the compactness method
also yields exact Sobolev regularity, but the converse need not hold (cf. \cite{BS99,FS01,S06,S10}).

Compactness of the $\bar\partial$-Neumann operator has attracted considerable attention, bringing together ideas from complex analytic geometry, functional analysis, partial differential equations, potential theory and spectral theory. A central objective is to characterize compactness in terms of the geometry and potential-theoretic properties of the boundary. Of particular interest is how degeneracies of the Levi form influence compactness, revealing connections between boundary geometry and the analytic behavior of the $\bar\partial$-Neumann problem.


The first broadly applicable potential-theoretic sufficient condition for
compactness is Catlin's Property $(P)$ \cite{C83}.  
A compact set $K\subset\mathbb C^n$ satisfies $(P)$ if, for every $M>0$, there are a
neighborhood $U_M$ of $K$ and a function
$\lambda_M\in C^2(U_M)$ such that $0\leq\lambda_M\leq1$ and the smallest eigenvalue of $\left(
\partial\bar\partial \lambda_M \right)_{n \times n}$ is at least $M$ on $U_M$. 
For a bounded pseudoconvex domain, Property $(P)$ of $b\Omega$ implies
compactness of $N_q$ for all $q \geq 1$ \cite{C83}.  McNeal introduced  Property
$(\widetilde P)$, replacing the uniform boundedness by a
self-bounded-gradient condition, and showed that 
Property $(\widetilde P)$ also implies 
compactness of $N_q$ \cite{M02}. 

Analytic varieties in the boundary provide the most visible geometric
obstruction.  If $b\Omega$ contains an analytic disc, 
 then it cannot satisfy $(P)$ 
\cite{FS01,S06}.  The converse direction, from absence of  
analytic discs to compactness, is much subtler.  On bounded convex
domains, Fu and Straube proved the sharp equivalence of compactness of $N_q$ with the absence of $q$-dimensional complex analytic varieties. 
In this setting the condition is also equivalent to Property $(P_q)$
\cite{FS98,FS01}.  
 For domains with bounded intrinsic geometry, Zimmer
characterized compactness by the absence of $q$-dimensional boundary
varieties \cite{Z21}.  On the other hand, Matheos constructed a smooth
bounded pseudoconvex complete Hartogs domain in $\mathbb C^2$ whose boundary
contains no analytic disc although $N_1$ is not compact \cite{M98}.  Thus the
absence of analytic discs is not sufficient in general.

Hartogs domains have played an important role in clarifying these relations.
Rotational symmetry converts the $\bar\partial$-Neumann problem into a
sequence of weighted problems, or equivalently into semiclassical electric
and magnetic Schr\"odinger operators \cite{FS02}.  Christ and Fu used this
correspondence and the Aharonov--Bohm effect to prove that, for smoothly
bounded complete Hartogs domains in $\mathbb C^2$, compactness is equivalent
to Property $(P)$ \cite{CF05}. 

The opposite implication---whether compactness rules out analytic
discs---was known in $\mathbb C^2$ (cf. Proposition 9  in \cite{FS01}).
The statement of the problem  in higher dimensions dates back at least to \cite{FS01} by Fu–Straube in 2001.
\c{S}ahuto\u{g}lu and Straube later state that,  {\it whether an analytic disc necessarily an obstruction to compactness is open} in \cite{SS06}. 
 In arbitrary
dimension, \c{S}ahuto\u{g}lu and Straube proved that a boundary complex
manifold obstructs compactness when the Levi form is nondegenerate in the
transverse complex directions, and, in particular, an analytic disc is an
obstruction if the Levi form has only one zero eigenvalue at some point of
the disc \cite{SS06}.  In $\mathbb C^3$, Dall'Ara obtained noncompactness
under the weaker hypothesis that a point of the boundary manifold has finite
regular D'Angelo $2$-type \cite{D18}. 

The purpose of this paper is to show that 
an analytic disc need not obstruct
compactness.  We construct a smoothly bounded pseudoconvex complete
Reinhardt domain $\Omega\subset\mathbb C^3$ for which $N_1$ is compact even
though $b\Omega$ contains a nonconstant analytic disc.  Since the disc forces
the failure of both $(P)$ and $(\widetilde P)$, the same example shows
that neither potential-theoretic property is necessary for compactness.
Together with Matheos's example \cite{M98},  in the general class
of smooth bounded pseudoconvex domains, the absence of non-constant
boundary analytic discs is neither necessary nor sufficient for
compactness of $N_1$.

Our construction belongs to the Hartogs--Fourier tradition but has a
different mechanism from the worm domain irregularity of
\cite{B92,C96}.  An infinitely flat layered plurisubharmonic potential is
designed so that its active layer moves to higher and higher scale near the
analytic disc.  A Hessian estimate and a radial integration by
parts argument force the relevant weighted Fourier spectral gaps to diverge.
This is the compact, rather than resonant, side of the
Schr\"odinger correspondence in \cite{FS02,CF05}.  We then prove compactness
of the boundary complex Green operator by a Fourier decomposition and use
the theorem of Raich and Straube that compactness of the boundary Green operator implies compactness of the interior $\bar\partial$-Neumann
operator \cite{RS08}.  This yields the following
main theorem.

\begin{theorem}
\label{thm:main}
There exists a bounded pseudoconvex complete Reinhardt domain $\Omega$ 
 in $\mathbb{C}^3$ with smooth boundary such that 
 $\Omega$ satisfies
\begin{enumerate}[label=\textup{(\roman*)}]
\item the \(\bar\partial\)-Neumann operator \(N_1\) on \(\Omega\) is compact;
\item there exists a nonconstant analytic disc in \(b\Omega\);
\item \(b\Omega\) fails Catlin's Property \((P_1)\) and McNeal's Property \((\tilde P_1)\).
\end{enumerate}
\end{theorem}

\section{The construction of the domain}
Set
$$ d(x)=\begin{cases}e^{-\frac1x},&x>0,\\0,&x=0,\end{cases}$$
and $$ G(x)=\int_0^x\frac{d(s)}s\,ds.$$
It is easy to verify that  \(G\in C^\infty([0,\infty))\) and every derivative of \(G\) vanishes at
zero. 
 Fixing \(\tau_2>0\), for \(j\ge2\),
set $ A_j=e^{-j^2}, \tau_{j+1}=\frac{\tau_j}{2j}$ and then
 define
\begin{equation}\label{eq:Phi}
 \Phi(z,w)=\sum_{j=2}^{\infty}A_jG(x_j)
\end{equation}
on \(\mathbb C^2\) with $x_j=\frac{|w|^2e^{j|z|^2}}{\tau_j}$.

\begin{lemma}
\label{lem:smooth}
The series \eqref{eq:Phi} converges in \(C^\infty\) on compact subsets of
\(\mathbb C^2\) to  the nonnegative 
plurisubharmonic function $\Phi$. Moreover,  $\Phi$ is positive when \(w\ne0\) and every derivative of
\(\Phi\) vanishes on \(\{w=0\}\).
\end{lemma}

\begin{proof}
It follows from the definition that $ \Phi(z,w) \geq 0$ and $ \Phi(z,w)>0$ whenever $w\ne0.$
%
%
%
%
For \(x\ge1\), we have
\[
 \begin{aligned}
 G(x)-\log x
 &=
 G(1)+\int_1^x\frac{e^{-\frac1s}-1}{s}\,ds.
 \end{aligned}
\]
Noting $ 0\le 1-e^{-\frac1s}\le \frac1s$ for $s\ge1$,
we obtain
$ G(x)=\log x+O(1)$
 as $x\to\infty$. 
 This implies that $ 0\le G(x)\le C\bigl(1+\log(1+x)\bigr)$ for any $x\ge0.$ 
Differentiating \(G'(x)=\frac{e^{-\frac1x}}{x}\) also yields that, for every
\(q\ge1\), $ |G^{(q)}(x)|\le C_qx^{-q}$, for $ x\ge1$.
Since $G$ is smooth 
 on \(0<x\le1\), we have
$ \sup_{x>0}\bigl|x^qG^{(q)}(x)\bigr|<\infty$ for all $ q\ge1.$ 
In particular, every derivative \(G^{(q)}\) is bounded on
\([0,\infty)\) for \(q\ge1\). 
Fix a compact bidisc $P_{R,T}:=\{(z,w)\in\mathbb C^2:|z|\le R,\ |w|\le T\}.$
%
On \(P_{R,T}\), $ 0\le x_j(z,w)\le T^2  \tau^{-1}_j e^{jR^2}.$
Let \(\alpha\) be a real multi-index of order
\(|\alpha|=m\ge1\). 
It follows that $ \sup_{P_{R,T}} \left|\partial^\alpha x_j\right| \le C_{\alpha,R,T}\,j^m  \tau^{-1}_j e^{jR^2}$ for some constant $C_{\alpha,R,T} >0$.
%
%
For $m\ge1$, by the chain rule, \(\partial^\alpha(G\circ x_j)\) is a finite
sum of terms of the form
$ G^{(q)}(x_j) \prod_{\nu=1}^{q}\partial^{\beta_\nu}x_j$,
where $ 1\le q\le m,  |\beta_\nu|\ge1,  \sum_{\nu=1}^{q}|\beta_\nu|=m.$
Because the derivatives \(G^{(q)}\) are bounded, it follows from the estimates above that 
$$ \sup_{P_{R,T}}  \left|\partial^\alpha(G(x_j))\right|  \le C_{m,R,T}\,  j^m\bigl(1+  \tau^{-1}_j e^{jR^2} \bigr)^m.$$
For \(m=0\), it follows that $$ \sup_{P_{R,T}}G(x_j) \le C_{R,T}\bigl(1+\log(1+ \tau^{-1}_j e^{jR^2})\bigr).$$
Consequently, for every multi-index \(\alpha\) of order \(m\ge0\),
$$ \sup_{P_{R,T}} \left|A_j \partial^\alpha G(x_j(z,w)) \right| \le C_{m,R,T}e^{-j^2}(1+j\log j)j^m \left(1+\tau^{-1}_j  e^{jR^2}\right)^m \leq e^{-\frac{j^2}{2}}$$
for $j$ sufficiently large. 
%
It then follows from the
 Weierstrass \(M\)-test  that, on every compact
subset of \(\mathbb C^2\), the series (\ref{eq:Phi}) defining \(\Phi\), together with
all of its derivatives, converges uniformly. Consequently, $ \Phi\in C^\infty(\mathbb C^2).$

For every \(z\in\mathbb C\), $ x_j(z,0)=0.$
Since every derivative of \(G\) vanishes at zero, the chain rule yields
$ \partial^\alpha G(x_j(z,0))=0$
for every multi-index \(\alpha\) and every \(j\ge2\). Because the
series of derivatives converges uniformly on compact sets, we may
differentiate term by term and obtain, for all $\alpha$, 
\[ \partial^\alpha\Phi(z,0) = \sum_{j=2}^{\infty}A_j \partial^\alpha G(x_j(z,0)) =0. \]

It follows from the straightforward calculation that,  on the set \(\{w\ne0\}\),
\begin{equation}\label{eq:compact-layer-Hessian}
\sqrt{-1} \partial\bar\partial G(x_j)
 =\sqrt{-1} j d(x_j)\,dz\wedge d\overline z
 +\sqrt{-1}\frac{d(x_j)}{x_j}
 \left(j\overline z\,dz+\frac{dw}{w}\right)
 \wedge
 \overline{\left(j\overline z\,dz+\frac{dw}{w}\right)} \geq 0.
\end{equation}
 Thus, by the term by term differentiation and 
 the continuity  of (\ref{eq:Phi}), $\Phi$ is a smooth plurisubharmonic function on $\mathbb{C}^2$.
%
%
%
\end{proof}

Let \(\vartheta\) be
a smooth convex increasing function
such that $\vartheta$ is zero on \((-\infty,0]\), has positive derivative on
\((0,\infty)\), and tends to infinity. Choosing \(R_0>0\), \(C>0\), set
\begin{equation*}
 \Psi(z,w)=\Phi(z,w)+C\vartheta(|z|^2+|w|^2-R_0^2).
\end{equation*}
Let 
\begin{equation*}
 \Omega=\{(z,w,u) \in \mathbb{C}^3 : |u|^2+\Psi(z,w)<1\}.
\end{equation*}

\begin{lemma}\label{bounded}
$\Omega$ is a bounded complete Reinhardt  domain in $\mathbb{C}^3$.
\end{lemma}

\begin{proof}
Since \(\vartheta(s)\to\infty\) as $s \to \infty$ and $\Phi \geq 0$, then  \(\Omega\) is bounded. For  $c >0$, let  \(B_c=\{(z,w)\in\mathbb C^2:\Psi(z,w)<c\}\).  \(\Psi(0,0)=0<c\), so the origin belongs to \(B_c\). 
Fix \((z,w)\in B_c\) and \(0\le \epsilon \le 1\). For every \(j\ge2\),
\[
 x_j(\epsilon z, \epsilon w)
 =\frac{ \epsilon^2|w|^2e^{j \epsilon^2|z|^2}}{\tau_j}
 \le
 \frac{|w|^2e^{j|z|^2}}{\tau_j}
 =x_j(z,w).
\]
 Since \(G\) is increasing and
\(A_j>0\), it follows that
\[
 \Phi(\epsilon z, \epsilon w)
 =\sum_{j\ge2}A_jG\bigl(x_j(\epsilon z, \epsilon w)\bigr)
 \le
 \sum_{j\ge2}A_jG\bigl(x_j(z,w)\bigr)
 =\Phi(z,w).
\]
Moreover,
since \(\vartheta\) is increasing,
\[
 \vartheta\bigl(|\epsilon z|^2+|\epsilon w|^2-R_0^2\bigr)
 \le
 \vartheta\bigl(|z|^2+|w|^2-R_0^2\bigr).
\]
Consequently,
$ \Psi(\epsilon z, \epsilon w) \le \Psi(z,w)<c,$
and thus \((\epsilon z, \epsilon w)\in B_c\).
This implies  that $B_c$ is star-shaped with center $(0, 0)$ and thus $B_c$ is connected.
Since each \(u\)-fiber is a  disc for any $(z, w) \in B_c$,
\(\Omega\) is connected.
It also follows from the definition that $\Omega$ is complete Reinhardt.
\end{proof}

\begin{proposition}
$\Omega$ is a bounded pseudoconvex complete Reinhardt  domain with smooth boundary. 
\end{proposition}

\begin{proof}
We claim that $\rho(z, w, u)=|u|^2+\Psi(z,w)-1$ is the defining function of $\Omega$. 
If \(u\ne0\), $\rho$ has nonzero gradient.  If \(u=0,w\ne0\), then $\rho$ has nonzero gradient since 
 \(\Psi \) is a strictly increasing function in $|w|^2$. 
If \(u=w=0\) and
\(\Psi=1\), then \(C\vartheta(|z|^2-R_0^2)=1\) and thus \(|z|>R_0\).
It follows that $\rho$ has nonzero gradient since 
 \(\Psi \) is a strictly increasing function in $|z|^2$ for  \(|z|>R_0\).
 Since $\rho$ is a smooth plurisubharmonic function on $\mathbb{C}^3$ by Lemma \ref{lem:smooth}, 
 $\Omega$ is a bounded pseudoconvex domain with smooth boundary. 
\end{proof}

\begin{proposition}
There exists an analytic disc on $b\Omega$ and thus $b\Omega$ fails Catlin’s Property \((P_1)\) and McNeal’s Property \((\tilde P_1)\).
\end{proposition}

We only construct an analytic disc in $b\Omega$. It is well known to experts that the existence of an analytic disc would be the obstruction to Catlin’s Property (\(P_1\)) and McNeal’s Property (\(\tilde P_1 \)) (for detailed proof, the reader may refer to \cite{FS01} for the former and page 1462 in \cite{S06} for the latter).

\begin{proof}
For every \(0<r<R_0\) and every \(|u_0|=1\), define $H( \zeta)=(\zeta,0,u_0)$
for $|\zeta|<r$. 
Because \(\Phi(\zeta, 0)=0\) and $\vartheta(|\zeta|^2-R_0^2)$=0 for all  $|\zeta|<r$, $H( \zeta) \in b\Omega$.
This implies that  $b\Omega$  contains an analytic disc.
%
\end{proof}

\section{Compactness of the $\bar\partial$-Neumann operator $N_1$}
\label{sec:green-compactness}

\subsection{Estimates near the weakly pseudoconvex set}

For \(t=|w|^2>0\), let
$
 W_j(t,z):=A_j d(x_j)
 =e^{-j^2}e^{-\tau_j t^{-1} e^{-j|z|^2}},$
and
$ S(t,z):=\sum_{j\geq2}W_j(t,z).$
Since \(\Phi, \Psi\) is radial in the \(w\)-variable, define its radial
representative by
$ \widehat{\Phi}(z,t) := \sum_{j=2}^{\infty} A_jG\left(\tau^{-1}_j t e^{j|z|^2}\right)$ and 
$\widehat{\Psi}(z,t) :=\widehat{\Phi}(z,t) +C\vartheta(|z|^2+t-R_0^2)$. 
%
Thus
$ \Phi(z,w)=\widehat{\Phi}(z,|w|^2),  \Psi(z,w)=\widehat{\Psi}(z,|w|^2)$. 
In what follows, we use the notation
$ \Phi_t(z,t) := \frac{\partial\widehat{\Phi}}{\partial t}(z,t), \Psi_t(z,t) := \frac{\partial\widehat{\Psi}}{\partial t}(z,t).$
%
It follows from term by term differentiation that 
$ S(t,z)=t\Phi_t(z,t).$
For fixed \(t>0\) and \(z\in\mathbb C\), 
 let \(J(t, z)\) be the first index at which \(W_j(t,z)\) is
maximal. Since \(0<W_j(t,z)\le e^{-j^2},\) $J(t, z)$ always exists as a positive integer.
%
Let $R(x)=\frac{d(x)}x$ for $x >0$.  Then \(R(x)>0\) for any \(x \in (0,\infty)\). 

\begin{lemma}
For every \(R>0\), there exist constants
\(c_R, C_R>0\) and \(t_R>0\) such that
\begin{align}
c_R\frac{\log(1/t)}{\log\log(1/t)}
&\le J(t,z)\le
C_R\frac{\log(1/t)}{\log\log(1/t)}
 \label{eq:compact-J-asymptotic}\\
 \sum_{j\geq2}jW_j(t,z)&\geq c_RJ(t,z)S(t,z),
 \label{eq:compact-mean-J}\\
 \sum_{j\geq2}A_jR(x_j)&\geq c_RS(t,z),
 \label{eq:compact-R-lower}\\
 \frac{\sum_{j\geq2}jA_jR(x_j)}
      {\sum_{j\geq2}A_jR(x_j)}
 &\leq C_RJ(t,z),
 \label{eq:compact-R-mean}
\end{align}
hold for \(|z|\le R,  0<t<t_R.\)
In particular,
\begin{equation}\label{eq:compact-J-limits}
 J(t,z)\rightarrow\infty,\quad
 tJ(t,z)^2\rightarrow0
\end{equation}
as $t \to 0$, 
uniformly for \(|z|\leq R\).
\end{lemma}

\begin{proof}
Write
$ b_j(z):=\tau_je^{-j|z|^2}, y_j:=\frac{b_j(z)}{t},  L_j:=2je^{|z|^2}.$
Then \(x_j=y^{-1}_j\), \(y_{j+1}=\frac{y_j}{L_j}\), and $ W_j=e^{-j^2-y_j}.$
Moreover, let 
\begin{equation}\label{eq:compact-ratio}
 r_j:=\log\frac{W_{j+1}}{W_j}
 =-(2j+1)+y_j(1-L_j^{-1}).
\end{equation}
It is easy to verify that the sequence \(r_j\) is strictly decreasing. Thus 
the sequence
\(W_j\) first increases and then decreases.  Since \(J\) is the smaller maximizing index
when a tie occurs, $ r_{J-1}\geq 0,  r_J\leq0.$

We first show that \(J(t,z)\to\infty\) uniformly for \(|z|\leq R\)
as \(t\to 0\). For every fixed \(j\geq2\),
\[
 y_j=\frac{\tau_je^{-j|z|^2}}{t}
 \geq\frac{\tau_je^{-jR^2}}{t},
\]
and therefore
\[
 r_j
 \geq
 -(2j+1)
 +\frac34\,\frac{\tau_je^{-jR^2}}{t}
 \rightarrow\infty
\]
uniformly for \(|z|\leq R\). Thus, for every fixed \(J_0\),
after decreasing \(t_R\), we have \(r_j>0\) for
\(2\leq j<J_0\), and consequently \(J(t,z)\geq J_0\).

For \(|z|\leq R\), $ 2j\leq L_j\leq2je^{R^2}.$
Combining with (\ref{eq:compact-ratio}), we reach 
\begin{equation}\label{eq:compact-active-y}
 c_R\leq y_J\leq C_RJ,\quad
 y_{J-1}\geq c_RJ,\quad
 \frac{c_R}{J}\leq x_J\leq C_R.
\end{equation}
We now show that  summands far from \(J\) are negligible.  If
\(h\geq2\), then
\(y_{J+h}\leq \frac{ y_J}{\prod_{\nu=J}^{J+h-1}2\nu}\).  
Combining with \eqref{eq:compact-ratio}, we obtain 
\begin{equation}\label{eq:compact-right-tail}
 W_{J+h}\leq W_{J+1}
 e^{-c_R(h-1)J-c_R(h-1)^2}.
\end{equation}
On the other hand, 
\[
 y_{J-h}=y_{J-1}\prod_{\nu=J-h}^{J-2}L_\nu
 \geq c_RJ\prod_{\nu=J-h}^{J-2}2\nu.
\]
For \(2\leq h\leq \frac{J}{2}\), it follows that 
\(y_{J-h}\ge c_RJ^h.\)
For
\(\frac{J}{2}<h\leq J-2\), it is also easy to see that \(y_{J-h}\ge c_RJ^{J/2}.\)
Both cases imply
$ \inf_{2\leq h\leq J-2} \frac{y_{J-h}}{hJ} \rightarrow\infty$ as 
$J\to\infty$.
It then follows from
$ \log\frac{W_{J-h}}{W_J} =J^2-(J-h)^2+y_J-y_{J-h}$
that 
\begin{equation*}
 \bigl(1+(J-h)+y_{J-h}+(J-h)y_{J-h}\bigr)W_{J-h}
 \leq C_Re^{-c_RhJ}W_J.
\end{equation*}
Set \(j=J-h\) and \(y=y_{J-h}\). Since
$ 1+j+y+jy=(1+j)(1+y)$
and \(y_{J-h}\geq c_RJ^2\) for \(h\geq2\), we have, for all sufficiently
large \(J\),
$\log\bigl((1+j)(1+y)\bigr)\leq \frac{y}{4}.$
Moreover, \(y_J\leq C_RJ\), and hence
$2Jh-h^2+y_J\leq A_RhJ$
for some \(A_R>0\). Since
$\inf_{2\leq h\leq J-2}\frac{y_{J-h}}{hJ}\rightarrow\infty,$
we may choose \(J_R\) so that, whenever \(J\geq J_R\),
$A_RhJ-\frac34y_{J-h}\leq-c_RhJ$
for every \(2\leq h\leq J-2\). Therefore,
\[
\begin{aligned}
\log\left(
\frac{
\bigl(1+(J-h)+y_{J-h}+(J-h)y_{J-h}\bigr)W_{J-h}}
{W_J}
\right)
&\leq -c_RhJ.
\end{aligned}
\]
%
Therefore 
\[
\begin{aligned}
\sum_{h=2}^{J-2}
 \bigl(1+(J-h)+y_{J-h}+(J-h)y_{J-h}\bigr)W_{J-h} \le
C_RW_J\sum_{h=2}^{J-2}e^{-c_RhJ}
\le C_Re^{-c_RJ}W_J,
\end{aligned}
\]after changing \(c_R\) and \(C_R\).
It follows from (\ref{eq:compact-right-tail}) that 
\[
\begin{aligned}
&\sum_{h=2}^{\infty}
 \bigl(1+(J+h)+y_{J+h}+(J+h)y_{J+h}\bigr)W_{J+h}\\
&\quad
\le
C_RW_{J+1}
\sum_{h=2}^{\infty}
(J+h+1)e^{-c_R(h-1)J-c_R(h-1)^2}\\
&\quad\le
C_Re^{-c_RJ}W_{J+1}.
\end{aligned}
\]
It then follows that
\begin{equation*}
\sum_{|j-J|\ge2}W_j
\le \sum_{|j-J|\geq2}(1+j+y_j+jy_j)W_j
 \leq C_Re^{-c_RJ}(W_J+W_{J+1})
\end{equation*}
and thus 
\begin{equation}\label{eq:compact-active-band}
 S(t,z)\asymp_R W_{J-1}+W_J+W_{J+1}\asymp_R W_J, 
\end{equation}
where $A\asymp_R B$ means that $A, B$ are comparable up to positive constants depending on $R$. This implies  \eqref{eq:compact-mean-J}. 
It follows from \eqref{eq:compact-active-y} that 
$$ c_R\leq\frac{b_J}{t}\leq C_RJ.$$
Consequently, there exists \(A_R>0\) such that
\[ \left| \log\frac1t-\log\frac1{b_J} \right| \le A_R\log J \]
holds for all sufficiently large \(J\). 
On the other hand, using  $ b_J=\tau_Je^{-J|z|^2}$ and 
$ \tau_J = \frac{\tau_2}{2^{J-2}(J-1)!},$
It follows from Stirling's formula  $ \log((J-1)!) = J\log J+O(J)$ that
$$  \left| \log\frac{1}{b_J} -  J\log J\right| \leq A'_R J.$$
Therefore, $$  \left| \log\frac{1}{t} -  J\log J\right| \leq A''_R J.$$
One can verify that this implies \eqref{eq:compact-J-asymptotic}.
Furthermore, since
$ A_jR(x_j)=A_j\frac{d(x_j)}{x_j}=y_jW_j,$ 
 \eqref{eq:compact-active-y} and
\eqref{eq:compact-active-band} yield
\[
 \sum_jA_jR(x_j)\geq y_JW_J\geq c_RW_J\geq c_RS,
\]
which proves \eqref{eq:compact-R-lower}. 
Combining 
$$ \sum_{j=J-1}^{J+1} j\,y_jW_j \le 2J\sum_{j=J-1}^{J+1} y_jW_j \le 2J \sum_{j=2}^{\infty} y_jW_j $$
and 
$$ \sum_{|j-J|
\geq 2}j\,y_jW_j \le \sum_{|j-J|
\geq 2}(1+j+y_j+jy_j)W_j \le C_Re^{-c_RJ}(W_J+W_{J+1}),$$
we obtain
\[
 \sum_jjA_jR(x_j)=\sum_jjy_jW_j
 \leq C_RJ\sum_jy_jW_j,
\]
which is \eqref{eq:compact-R-mean}.  Finally,
\eqref{eq:compact-J-limits} follows from
\eqref{eq:compact-J-asymptotic}.
We note that all constants  above depend only on \(R\). Since
\(J(t,z)\to\infty\) uniformly for \(|z|\leq R\) as \(t\to0\), after
choosing \(J_R\) we may decrease \(t_R\) so that
$J(t,z)\geq J_R$ 
for every \(|z|\leq R\) and \(0<t<t_R\). Thus the preceding 
estimates hold uniformly in \(z\) and in every  \(h\).
\end{proof}


\begin{lemma}\label{lem:compact-matrix}
For every \(R>0\), there exist \(t_0>0\) and \(c>0\) such that
\begin{equation*}
\left( \partial\bar\partial \Phi(z,w)\right)_{2 \times 2}\geq c J(t,z)S(t,z)I_2,
\end{equation*}
for $ |z|\leq R,  0<t<t_0$.
\end{lemma}

\begin{proof}
Since $t=|w|^2>0$, we only need to prove the lemma for  \(w\neq0\). We apply Sylvester’s criterion to verify the positivity of the matrix  
$\left( \partial\bar\partial \Phi(z,w)\right)_{2 \times 2} - c_RJ(t,z)S(t,z)I_2$. 
By (\ref{eq:Phi}), (\ref{eq:compact-layer-Hessian}), (\ref{eq:compact-mean-J}), the first leading principal minor satisfies
$$ \sum_{j\ge2}jW_j +|z|^2 \sum_{j\ge2}j^2A_jR(x_j) -\frac {c_R}4 JS \ge \sum_{j\ge2}jW_j -\frac {c_R}4 JS \ge  \frac{3c_R}{4} JS \geq 0.
$$
Note that it follows from Cauchy--Schwarz that $$ \left( \sum_{j\ge2}j A_jR(x_j) \right)^2 \le \left( \sum_{j\ge2} A_jR(x_j) \right) \left(\sum_{j\ge2}j^2A_jR(x_j) \right) .$$
The determinant satisfies
\begin{align*}
& \left(\sum_{j\ge2}jW_j - \frac {c_R}4 JS \right)\left(\frac{\sum_{j\ge2}A_jR(x_j)}{t}-\frac {c_R}4 JS \right) \\
  &~~ +|z|^2\left[
      \left(\frac{\sum_{j\ge2} A_jR(x_j)}{t}- \frac {c_R}4 JS\right)\left( \sum_{j\ge2}j^2A_jR(x_j) \right)
      -\frac{\left( \sum_{j\ge2}j A_jR(x_j) \right)^2}{t}    \right]\\
 &= \left(\sum_{j\ge2}jW_j - \frac {c_R}4 JS \right)\left(\frac{\sum_{j\ge2}A_jR(x_j)}{t}-\frac {c_R}4 JS \right) \\
  &~~  +|z|^2\left[
      \left(\frac{\sum_{j\ge2} A_jR(x_j)}{t}- \frac {c_R}4 JS\right)\left( \sum_{j\ge2}j^2A_jR(x_j) -\frac{\left( \sum_{j\ge2}j A_jR(x_j) \right)^2}{\sum_{j\ge2} A_jR(x_j)}  \right)
      - \frac {c_R}4 JS\frac{\left( \sum_{j\ge2}j A_jR(x_j) \right)^2}{\sum_{j\ge2} A_jR(x_j)}    \right]\\
& \geq   \left(\sum_{j\ge2}jW_j - \frac {c_R}4 JS \right)\left(\frac{\sum_{j\ge2}A_jR(x_j)}{t}-\frac {c_R}4 JS \right) - \frac {c_R}4 |z|^2 JS\frac{\left( \sum_{j\ge2}j A_jR(x_j) \right)^2}{\sum_{j\ge2} A_jR(x_j)} \\
 &\ge  \frac {3{c_R}}4 JS \frac{\sum_{j\ge2} A_jR(x_j)}{2t} - \frac {c_R}4 JS \frac{\sum_{j\ge2} A_jR(x_j)}{t} \geq 0,
 \end{align*}
where the second inequality follows from (\ref{eq:compact-mean-J}), (\ref{eq:compact-R-lower}), (\ref{eq:compact-J-limits}).
\end{proof}

Consider
\[
 B:=\{(z,w)\in\mathbb C^2:\Psi(z,w)<1\}
\]
to be the base of the Hartogs domain $\Omega$.  On \(B\), define
\begin{equation*}
 \varphi(z,w):=-\frac12\log(1-\Psi(z,w))
\end{equation*}
and write 
\[
 \Omega=\{(z,w,u)\in\mathbb C^3 :(z,w)\in B,\ |u|<e^{-\varphi(z,w)}\}.
\]

\begin{proposition}
The weakly pseudoconvex set of \(M=b\Omega\) is
\begin{equation*}\label{eq:compact-weak-set}
 K:=\{(z,0,u):|z|\leq R_0,\ |u|=1\}.
\end{equation*}
\end{proposition}

\begin{proof}
Recall that $ \rho(z,w,u)=|u|^2+\Psi(z,w)-1$
is a smooth defining function for \(\Omega\). Set
\(s=|z|^2+|w|^2-R_0^2\). For \(p=(z,w,u)\) and
\(V=(\xi,\eta,\nu)\in\mathbb C^3\), the Levi form is
\[
 \mathcal L_\rho(p;V)
 =|\nu|^2+\mathcal L_\Phi((z,w);(\xi,\eta))
 +C\vartheta'(s)(|\xi|^2+|\eta|^2)
 +C\vartheta''(s)|\bar z\,\xi+\bar w\,\eta|^2.
\]

Let \(p\in K\). By Lemma~\ref{lem:smooth}, all derivatives
of \(\Phi\) vanish at \(w=0\). 
Consequently, $ \partial\Psi(z,0)=0,  \partial\bar\partial\Psi(z,0)=0$ and thus $\partial\rho(p)=\bar u\,du.$ 
Moreover, 
$ \vartheta(s)=\vartheta'(s)=\vartheta''(s)=0$ for $s\leq0$.
Since \(|u|=1\), the complex tangent space is
$ T_p^{1,0}M=\{(\xi,\eta,\nu)\in\mathbb C^3:\nu=0\}.$
Hence \(\mathcal L_\rho(p;V)=0\) for every
\(V\in T_p^{1,0}M\), so every point of \(K\) is weakly
pseudoconvex.

Conversely, let \(p\in M\setminus K\). If \(w\neq0\),
then it follows from \eqref{eq:Phi} and \eqref{eq:compact-layer-Hessian}
that
\[
 \mathcal L_\Phi((z,w);(\xi,\eta))
 =
 \sum_{j\geq2}A_j
 \left[
 j\,d(x_j)|\xi|^2
 +\frac{d(x_j)}{x_j}
 \left|j\bar z\,\xi+\frac{\eta}{w}\right|^2 
 \right]>0
\]
whenever \((\xi,\eta)\neq(0,0)\). 
It follows that
\(\mathcal L_\rho(p;V)>0\) for every nonzero \(V\in\mathbb C^3\).
%
If \(w=0\), then necessarily \(|z|>R_0\). 
Thus \(s>0\) and
\(\vartheta'(s)>0\). Since \(\Phi\) is plurisubharmonic,
\[
 \mathcal L_\rho(p;V)
 \geq |\nu|^2+C\vartheta'(s)(|\xi|^2+|\eta|^2)>0
\]
for  all $V\neq0$.
Therefore
the weakly pseudoconvex set of \(M\) is exactly
the set \(K\).
\end{proof}

\begin{proposition}
\begin{equation}\label{eq:compact-varphi-ratio}
\left( \partial \bar\partial \varphi \right)_{2\times 2}\geq c_RJ(t,z)(t\varphi_t)I_2
\end{equation}
holds in a neighborhood of $K$.
\end{proposition}

\begin{proof}
Let $ K_0:=\{(z,0):|z|\leq R_0\}\subset B$
be the projection of \(K\) to the base. Since \(\Psi=0\)
on \(K_0\), we may  choose
\(R>R_0\) and \(\varepsilon>0\) such that
$ 0\leq\Psi(z,w)\leq\frac12$ for 
$|z|\leq R$ and $ |w|^2\leq\varepsilon$.
Let $ U:=\{(z,w):|z|<R,\ |w|^2<\varepsilon\} \Subset B$
be a neighborhood of \(K_0\). All subsequent choices of
\(\varepsilon\) are uniform for \(|z|\leq R\).
It follows from \eqref{eq:compact-J-limits} that 
$$\sqrt{-1} C \partial\bar\partial \vartheta
 = \sqrt{-1} C\left( \vartheta'(s)(dz \wedge d\bar z +d w\wedge d\bar w) +\vartheta''(s)\,
    \partial s\wedge \overline{\partial s} \right)
  \geq \sqrt{-1}c_RJ(t,z)\,Ct\vartheta'(s)(dz \wedge d\bar z +d w\wedge d\bar w).$$
By Lemma~\ref{lem:compact-matrix} and $ t\Phi_t(z,t) =S(t,z)$, we have 
\[
 \sqrt{-1}  \partial\bar\partial  \Psi
 =\sqrt{-1} \partial\bar\partial  \Phi+ \sqrt{-1} C \partial\bar\partial \vartheta
  \geq  \sqrt{-1}  c_R t J(t,z) \Psi_t (dz \wedge d\bar z +d w\wedge d\bar w).
\]
Therefore,
\begin{equation*}
\begin{split}
 \sqrt{-1}  \partial\bar\partial  \varphi
& =\frac{ \sqrt{-1}  \partial\bar\partial  \Psi}{2(1-\Psi)}
  +\frac{ \sqrt{-1} \partial\Psi\wedge\overline{\partial\Psi}}
         {2(1-\Psi)^2} 
\geq\frac{\sqrt{-1}  \partial\bar\partial \Psi}{2(1-\Psi)} \\
& \geq   \sqrt{-1} c_RJ(t,z)\frac{t\Psi_t}{2(1-\Psi)} (dz \wedge d\bar z +d w\wedge d\bar w) \\
& = \sqrt{-1} c_RJ(t,z)(t\varphi_t)(dz \wedge d\bar z +d w\wedge d\bar w).
 \end{split}
 \end{equation*}
This proves \eqref{eq:compact-varphi-ratio} on
\(U\cap\{w\neq0\}\).

Furthermore,
\[
0\le J(t,z)t\varphi_t(z,t)
\le C_R\bigl(J(t,z)e^{-J(t,z)^2}+tJ(t,z)\bigr)
\rightarrow0
\]
uniformly for \(|z|\le R\) as \(t\to0\). We therefore define
\(J(t,z)t\varphi_t(z,t)=0\) when \(t=0\). With this convention, the
right-hand side of \eqref{eq:compact-varphi-ratio} extends continuously
to \(w=0\), where the inequality follows from
\(\partial\bar\partial\varphi\ge0\).
\end{proof}

\subsection{Compactness}

For a \((0,1)\)-form $U=U_1\,d\overline z+U_2\,d\overline w \in L^2_{(0,1)}(B,e^{-2n\varphi})$
on the base $B$, write
\[
 \|U\|_{2n\varphi}^2
 :=\int_B|U|^2e^{-2n\varphi}\,dV.
\]
Let $\operatorname{Dom}Q_{n,\varphi}=\operatorname{Dom}\bar\partial\cap \operatorname{Dom}\bar\partial_{2n\varphi}^{*} \cap L^2_{(0,1)}(B,e^{-2n\varphi})$
and define \[ Q_{n,\varphi}(U) = \|\bar\partial U\|_{2n\varphi}^{2} + \|\bar\partial_{2n\varphi}^{*}U\|_{2n\varphi}^{2} \] for $U \in \operatorname{Dom}Q_{n,\varphi}$. 
It then follows from the
Morrey--Kohn--Hörmander identity that, 
\begin{equation}\label{eq:compact-Morrey-form}
 Q_{n,\varphi}(U)
 =\sum_{a,b=1}^2
   \left\|\frac{\partial U_a}{\partial\overline z_b}\right\|_{2n\varphi}^2
 +2n \int_B \sum_{a,b=1}^2  \frac{\partial^2 \varphi}{\partial z_a \partial \bar z_b} U_a \overline{U_b}
       e^{-2n\varphi}\,dV
\end{equation}
hold for any $U  \in \operatorname{Dom}Q_{n,\varphi}$ with compact support in $B$, 
where \(z_1=z\) and \(z_2=w\).  

\begin{lemma}
\label{lem:compact-radial}
Let $P$ be an open neighborhood of $K_0$ in the base and $c>0$. Let \(\varphi(z,w)=\varphi(z,|w|^2)\) be such that $\sqrt{-1}\partial\bar\partial \varphi$ is 
positive definite for \(w\neq0\). If for every \(M>0\), there is
\(\eta_M>0\), with \(M\eta_M\) as small as desired, such that
\begin{equation}\label{eq:compact-radial-hypothesis}
\left( \partial\bar\partial \varphi\right)_{2 \times 2} \geq cM t\varphi_t I_2
\end{equation}
holds on $P \cap \{0<t<2\eta_M\}$, 
then, for every \(A>0\), there exists \(n_A\) such that
\begin{equation*}\label{eq:compact-local-gap}
 Q_{n,\varphi}(U)\geq A\|U\|_{2n\varphi}^2
\end{equation*}
holds for every \( U\in\operatorname{Dom}Q_{n,\varphi}\) with support in \(P' \Subset P\) and 
$n\geq n_A$.
\end{lemma}

\begin{proof}
By density lemma (cf. Proposition 2.3 on page 15 in \cite{S10}), 
it suffices to consider smooth forms compactly supported in \(P\).
Fix an open set \(P_0\) such that $ \overline{P'}\subset P_0\Subset P$ with 
$ \delta=\operatorname{dist} \bigl(\overline{P_0},\mathbb C^2\setminus P\bigr)>0.$
We first consider a smooth $(0, 1)$-form supported in $P_0$. 
Fix \(A>0\), and choose $ M\geq \max\left\{1,\frac{4A}{c}\right\}.$
Choose \(\eta=:\eta_M\) small enough so that 
$ 4cM\eta\leq1,  8\eta<\delta^2$ and let  \(T=2\eta\).
%
Set
\[
 Z_\eta:=
 \left\{z:\text{there exists }w_0
 \text{ with }(z,w_0)\in\overline{P_0},
 \ |w_0|^2\leq T\right\}.
\]
For \(z\in Z_\eta\) and \(|w|^2\leq T\),
$ |(z,w)-(z,w_0)| \leq |w|+|w_0| \leq2\sqrt T =\sqrt{8\eta}<\delta.$
Consequently,
$ Z_\eta\times\{w:|w|^2\leq T\}\subset P.$
In particular, rotating a form supported in
\(P_0\cap\{t<T\}\) does not move its support outside \(P\).
Fix \(z\in Z_\eta\). 
Since
$ \frac{d}{dt}\bigl(t\varphi_t(z,t)\bigr) =\varphi_t(z,t)+t\varphi_{tt}(z,t) =\varphi_{w\bar w}(z,w)\geq 0$
and $ \lim_{t\to 0^+}t\varphi_t(z,t)=0,$
we have $ t\varphi_t(z,t)  \geq0$ and thus
$ \frac{d }{d t}e^{-2n\varphi(z,t)}=-2n\varphi_t(z,t)e^{-2n\varphi(z,t)} \leq0.$

Consider first
$ f(w)=w^kh(t)$ for $ k\geq0$ with $h(t)=0$ near $T$.
Since the boundary terms vanish, and it follows from the fundamental theorem in calculus that 
\begin{align*}
 0
 &=
 \left[t^{k+1}|h|^2 e^{-2n\varphi(z,t)}\right]_0^T \notag\\
 &=(k+1) \int_0^Tt^k|h|^2 e^{-2n\varphi(z,t)}\,dt
   +2\operatorname{Re}\int_0^T
          t^{k+1}h'\overline h e^{-2n\varphi(z,t)}\,dt- 2n\int_0^T(t\varphi_t)t^k|h|^2 e^{-2n\varphi(z,t)}\,dt.
 \label{eq:compact-radial-ibp}
\end{align*}
By Cauchy--Schwarz inequality,
\begin{align*}
 \left|\int_0^Tt^{k+1}h'\overline h e^{-2n\varphi(z,t)}\,dt\right|
 &\leq
 \left(\int_0^Tt^{k+1}|h'|^2 e^{-2n\varphi(z,t)}\,dt\right)^{1/2}
 \left(\int_0^Tt^{k+1}|h|^2 e^{-2n\varphi(z,t)}\,dt\right)^{1/2}\\
 &\leq \sqrt{T} \left(\int_0^Tt^{k+1}|h'|^2 e^{-2n\varphi(z,t)}\,dt \right)^{\frac12} \left( \int_0^Tt^k|h|^2 e^{-2n\varphi(z,t)}\,dt \right)^{\frac12} .
\end{align*}
Therefore,
\begin{align*}
 (k+1) \int_0^Tt^k|h|^2 e^{-2n\varphi(z,t)}\,dt
 &\leq2\sqrt{T} \left(\int_0^Tt^{k+1}|h'|^2 e^{-2n\varphi(z,t)}\,dt \right)^{\frac12} \left( \int_0^Tt^k|h|^2 e^{-2n\varphi(z,t)}\,dt \right)^{\frac12}  \\
 &~~~+ 2n\int_0^T(t\varphi_t)t^k|h|^2 e^{-2n\varphi(z,t)}\,dt \\
 &\leq\frac{k+1}{2} \int_0^Tt^k|h|^2 e^{-2n\varphi(z,t)}\,dt +\frac{2T}{k+1}\int_0^Tt^{k+1}|h'|^2 e^{-2n\varphi(z,t)}\,dt \\
 &~~~+2n\int_0^T(t\varphi_t)t^k|h|^2 e^{-2n\varphi(z,t)}\,dt.
\end{align*}
Using \(2cMT=4cM\eta\leq1\),
we have 
\begin{align*}
\frac{cM}{2} \int_0^Tt^k|h|^2 e^{-2n\varphi(z,t)}\,dt 
& \leq \frac{cM(k+1)}2 \int_0^Tt^k|h|^2 e^{-2n\varphi(z,t)}\,dt \\
 &\leq\frac{2cMT}{k+1}\int_0^Tt^{k+1}|h'|^2 e^{-2n\varphi(z,t)}\,dt+cM 2n\int_0^T(t\varphi_t)t^k|h|^2 e^{-2n\varphi(z,t)}\,dt \\
 &\leq \int_0^Tt^{k+1}|h'|^2 e^{-2n\varphi(z,t)}\,dt+cM 2n\int_0^T(t\varphi_t)t^k|h|^2 e^{-2n\varphi(z,t)}\,dt .
\end{align*}

Now consider $ f(w)=\bar w^{\,m}h(t)$ for $m\geq1$.
Then $ \frac{\partial f}{\partial\bar w} =\bar w^{\,m-1}(mh+th').$
It follows from $ t^mh(t) =\int_0^t s^{m-1}\bigl(mh(s)+sh'(s)\bigr)\,ds$ 
and Cauchy--Schwarz inequality that 
\begin{align*}
 t^{2m}|h(t)|^2
 &\leq
 \left(\int_0^ts^{m-1}\,ds\right)
 \left(\int_0^ts^{m-1}|mh+sh'|^2\,ds\right)=\frac{t^m}{m}
   \int_0^ts^{m-1}|mh+sh'|^2\,ds.
\end{align*}
Then, by the monotonicity of $e^{-2n\varphi(z,t)}$ in $t$,  we obtain
\begin{align*}
 &~~~ \int_0^Tt^m|h|^2 e^{-2n\varphi(z,t)}\,dt \\
 &\leq\frac1m
 \int_0^T e^{-2n\varphi(z,t)}\,
       \int_0^ts^{m-1}|mh+sh'|^2\,ds\,dt \notag\\
 &=\frac1m
 \int_0^T s^{m-1}|mh+sh'|^2
        \left( \int_s^T e^{-2n\varphi(z,t)}dt \right)\,ds \notag\\
 &\leq\frac1m
 \int_0^Ts^{m-1}|mh+sh'|^2(T-s) e^{-2n\varphi(z,s)} \,ds \notag\\
 &\leq\frac{T}{m} \int_0^Tt^{m-1}|mh+th'|^2 e^{-2n\varphi(z,t)}\,dt .
\end{align*}
Thus
\begin{align*}
&~~~ \int_0^Tt^{m-1}|mh+th'|^2 e^{-2n\varphi(z,t)}\,dt +cM 2n\int_0^T(t\varphi_t)t^m|h|^2 e^{-2n\varphi(z,t)}\,dt  \\
 & \geq \int_0^Tt^{m-1}|mh+th'|^2 e^{-2n\varphi(z,t)}\,dt  \geq\frac{m}{T} \int_0^Tt^m|h|^2 e^{-2n\varphi(z,t)}\,dt \\
& =\frac{m}{2\eta} \int_0^Tt^m|h|^2 e^{-2n\varphi(z,t)}\,dt  \geq\frac{cM}{2} \int_0^Tt^m|h|^2 e^{-2n\varphi(z,t)}\,dt.
\end{align*}
%
In both cases, using polar coordinates, we get 
\begin{align*}
 &\int_{|w|^2<T}
 \left(
 \left|\frac{\partial f}{\partial\bar w}\right|^2
 +2ncM(t\varphi_t)|f|^2
 \right)e^{-2n\varphi(z,t)}\,dA(w) 
 \geq
 \frac{cM}{2}
 \int_{|w|^2<T}|f|^2 e^{-2n\varphi(z,t)}\,dA(w).
\end{align*}
Integrating in \(z\), we obtain
\begin{equation}\label{alice}
 \left\|\frac{\partial f}{\partial\bar w}
       \right\|_{2n\varphi}^2
 +2ncM\int_B(t\varphi_t)|f|^2e^{-2n\varphi}\,dV
 \geq\frac{cM}{2}\|f\|_{2n\varphi}^2.
\end{equation}

Let $U$ be a smooth $(0, 1)$-form supported in  $P_0\cap\{t<T\}$. 
Write \(w=re^{i\theta}\) and expand the coefficients in Fourier series
$U_1=\sum_{k\in\mathbb Z}a_k(z,r)e^{ik\theta}, U_2=\sum_{k\in\mathbb Z}b_k(z,r)e^{ik\theta}.$
Write $U=\sum_{\ell\in\mathbb Z}U^{(\ell)}$, where $U^{(\ell)} = a_\ell e^{i\ell\theta}\,d\bar z + b_{\ell+1}e^{i(\ell+1)\theta}\,d\bar w$ has angular character \(\ell\), namely, $R_\alpha^*U^{(\ell)}=e^{i\ell\alpha}U^{(\ell)}$
 under the rotation $R_\alpha(z,w)=(z,e^{i\alpha}w)$. 
Since \(B\) and \(e^{-2n\varphi}\) are independent of \(\theta\), it follows that
\[ \|U\|_{2n\varphi}^{2} =\sum_{\ell\in\mathbb Z}\|U^{(\ell)}\|_{2n\varphi}^{2}.\]
Also, since $\bar\partial$ and $\bar\partial_{2n\varphi}^{*}$ are independent of $\theta$, 
$$Q_{n,\varphi}(U)=\sum_{\ell\in\mathbb Z}Q_{n,\varphi}(U^{(\ell)}).$$
It follows from (\ref{eq:compact-Morrey-form}), (\ref{eq:compact-radial-hypothesis}), (\ref{alice}) that 
\begin{align*}
 Q_{n,\varphi}(U^{(\ell)})
 &\geq
 \sum_{a=1}^2
 \left[
 \left\|\frac{\partial U_a^{(\ell)}}{\partial\bar w}
       \right\|_{2n\varphi}^2
 +2ncM\int_B(t\varphi_t)|U_a^{(\ell)}|^2
                  e^{-2n\varphi}\,dV
 \right] \geq\frac{cM}{2}
       \sum_{a=1}^2\|U_a^{(\ell)}\|_{2n\varphi}^2,
\end{align*}
yielding
\[
 Q_{n,\varphi}(U)
 \geq\frac{cM}{2}\|U\|_{2n\varphi}^2
 \geq2A\|U\|_{2n\varphi}^2.
\]

For a general smooth $(0, 1)$-form $U$ supported in $P'$, 
choose cutoff functions $\chi_0, \chi_1$ invariant under rotations in $w$ satisfying
$ \chi_0^2+\chi_1^2=1$ with
$ \operatorname{supp}(\chi_0U)\subset\{t\leq \frac{3\eta}2\}$
and $ \operatorname{supp}(\chi_1U)\subset\{t\geq\eta\}.$
Moreover,
$ \sum_{j=0}^1|\bar\partial\chi_j|^2 \leq\frac{C  }{\eta}
       \mathbf1_{\{\eta\leq t\leq2\eta\}}$
for some $ C>0 $.
Choose $\lambda>0$ such that $\lambda$ is less than the infimum of the smallest eigenvalue of $\left(\partial\bar\partial \varphi \right)_{2 \times 2}$ on $P_0 \cap \{|w|^2\geq\eta\}$.
%
 Then (\ref{eq:compact-Morrey-form}) implies 
\[
 Q_{n,\varphi}(\chi_1U)
 \geq2n\lambda\|\chi_1U\|_{2n\varphi}^2.
\]
It follows from
$ \sum_{j=0}^1\chi_j^2=1$ and  $ \sum_{j=0}^1\chi_j\bar\partial\chi_j=0$
that 
\begin{align*}
 \sum_{j=0}^1\sum_{a,b=1}^2
 \left\|
 \frac{\partial(\chi_jU_a)}{\partial\bar z_b}
 \right\|_{2n\varphi}^2
 =
 \sum_{a,b=1}^2
 \left\|\frac{\partial U_a}{\partial\bar z_b}
 \right\|_{2n\varphi}^2 +
 \int_B\sum_{j=0}^1|\bar\partial\chi_j|^2
            |U|^2e^{-2n\varphi}\,dV.
\end{align*}
and
\[ \sum_{j=0}^1 \sum_{a,b=1}^2 \varphi_{z_a\bar z_b} (\chi_jU_a)\overline{\chi_jU_b} = \sum_{a,b=1}^2 \varphi_{z_a\bar z_b}U_a\overline{U_b}.\]
Consequently,
\[
 Q_{n,\varphi}(\chi_0U)+Q_{n,\varphi}(\chi_1U)
 =Q_{n,\varphi}(U)+  \int_B\sum_{j=0}^1|\bar\partial\chi_j|^2
                   |U|^2e^{-2n\varphi}\,dV
\]
where
\begin{align*}
 0\leq \int_B\sum_{j=0}^1|\bar\partial\chi_j|^2
                   |U|^2e^{-2n\varphi}\,dV \leq\frac{C}{\eta}
 \int_{\{\eta\leq t\leq2\eta\}}
             |U|^2e^{-2n\varphi}\,dV.
\end{align*}
By \eqref{eq:compact-Morrey-form},
\[
 Q_{n,\varphi}(U)
 \geq2n\lambda
 \int_{\{\eta\leq t\leq2\eta\}}
             |U|^2e^{-2n\varphi}\,dV.
\]
Hence
\[
 \int_B\sum_{j=0}^1|\bar\partial\chi_j|^2
                   |U|^2e^{-2n\varphi}\,dV 
 \leq\frac{C}{2n\lambda\eta}
       Q_{n,\varphi}(U).
\]
By choosing $n_A$ sufficiently large, 
 we have 
$$ 
 \int_B\sum_{j=0}^1|\bar\partial\chi_j|^2      |U|^2e^{-2n\varphi}\,dV \leq Q_{n,\varphi}(U)$$
and 
$$ Q_{n,\varphi}(\chi_1U) \geq2A\|\chi_1U\|_{2n\varphi}^2, $$
for \(n\geq n_A\). 
Therefore,  we have 
\begin{align*}
 2Q_{n,\varphi}(U)
 &\geq Q_{n,\varphi}(U)+ \int_B\sum_{j=0}^1|\bar\partial\chi_j|^2
                   |U|^2e^{-2n\varphi}\,dV\\
 &=Q_{n,\varphi}(\chi_0U)
   +Q_{n,\varphi}(\chi_1U)\\
 &\geq2A\left(
       \|\chi_0U\|_{2n\varphi}^2
       +\|\chi_1U\|_{2n\varphi}^2\right)\\
 &=2A\|U\|_{2n\varphi}^2.
\end{align*}
\end{proof}

By \eqref{eq:compact-J-limits} and \eqref{eq:compact-varphi-ratio}, (\ref{eq:compact-radial-hypothesis}) holds. 
For any given \(M>0\), choose
\(\eta_M\) so small that \(J(t,z)\geq M\) for \(t<2\eta_M\), as well as \(M\eta_M\) being sufficiently small.
Define
\begin{equation}\label{eq:compact-gamma}
 \gamma_{n}(P')
 :=\inf\left\{Q_{n,\varphi}(U):
 \|U\|_{2n\varphi}=1,\ 
 U\in\operatorname{Dom}Q_{n,\varphi},\
 \operatorname{supp}U\subset P'\right\}.
\end{equation}
Lemma~\ref{lem:compact-radial} yields
\begin{equation}\label{eq:compact-gaps}
 \gamma_{n}(P')\rightarrow\infty
\end{equation}
as $n\to+\infty$.

\begin{proposition}\label{thm:compact-G1}
The complex Green operator \(G_1\) on \((0,1)\)-forms on \(M=b\Omega\)
is compact.
\end{proposition}

\begin{proof}
Let $ M^\circ:=M\cap\{u\neq0\}$ and $\pi(z,w,u)=(z,w)$ be a holomorphic projection from $\Omega$ to the base $B$.
Let  a smooth diffeomorphism $ \mathcal F: B\times\mathbb S^1\rightarrow M^\circ$ be given by
$ \mathcal F(z,w,\theta) =(z,w,e^{-\varphi(z,w)}e^{i\theta}).$
Choose open sets $ K_0\subset V_0\Subset V_1\Subset V_2\Subset P'$. Let $\psi_0: \mathbb{C} \to [0,1]$ be a 
 smooth function  with compact support in $V_1$ and $\kappa_0: \mathbb{C}^2 \to [0,1]$ be a smooth  function  with compact support in $V_2$
such that $ \psi_0=1$ near $\overline{V_0}$ and $ \kappa_0=1$  near $\overline{V_1}.$
Extend these functions by zero to the base $B$ and set
$ \psi=\psi_0\circ\pi,  \kappa=\kappa_0\circ\pi.$ 
Let  $ c_0:=\min_{\overline{V_2}}e^{-\varphi}>0$. Then
$ \mathcal F(\overline{V_2}\times\mathbb S^1) \subset M\cap\{|u|\geq c_0\}.$
Write \(z_1=z\), \(z_2=w\). On \(M^\circ\), the CR tangential vector
fields and their conjugates are
\[
 \overline L_a
 =\frac{\partial}{\partial\overline z_a}
   -i\varphi_{\overline z_a}\frac{\partial}{\partial\theta},
 \quad
 L_a
 =\frac{\partial}{\partial z_a}
   +i\varphi_{z_a}\frac{\partial}{\partial\theta},
 \quad a=1,2.
\]
Let \(\overline\omega_a\) be the restriction of
\(d\overline z_a\) to \(T^{0,1}M\). Then
$ \overline\omega_a(\overline L_b)=\delta_{ab},  \overline\partial_b\overline\omega_a=0.$
Define a Hermitian metric \(h_{\mathrm p}\) on the
\((0,1)\)-cotangent bundle over \(M^\circ\) by
$ h_{\mathrm p}(\overline\omega_a,\overline\omega_b)=\delta_{ab}.$
Let \(h_{\mathrm E}\) be the Euclidean induced metric and
let $ h=(1-\kappa)h_{\mathrm E}+\kappa h_{\mathrm p}.$
Then $h$ defines a smooth
positive Hermitian metric on \(M\), and $h$ also induces Hermitian
 metrics for $(0, q)$-forms on $M$. 
Let $dV_{z,w}\,\frac{d\theta}{2\pi}$ be the standard smooth volume form on $B\times\mathbb S^1$ and 
define a smooth volume form on $M^\circ$ by $ \mathcal F^*d\mu_{\mathrm p}
 =dV_{z,w}\,\frac{d\theta}{2\pi}.$
If \(d\sigma\) denotes Euclidean surface volume form, define $ d\mu=(1-\kappa+\kappa r)\,d\sigma$ to be a smooth positive volume form on $M$, where
$r=\frac{d\mu_{\mathrm p}}{d\sigma}$ is a smooth density function on $M^\circ$. 
On a neighborhood of
\(\mathcal F(\overline{V_1}\times\mathbb S^1)\),
$ h=h_{\mathrm p}, d\mu=d\mu_{\mathrm p}.$
Note that both \(h\) and \(d\mu\) are invariant under
$ R_\alpha(z,w,u)=(z,w,e^{i\alpha}u).$

We will use \(|\cdot|, \|\cdot\|, (\cdot, \cdot)\) denotes the norm and inner product induced by
\(h,d\mu\), and \(\overline\partial_b^\dagger\) denotes the corresponding
 adjoint of $\overline\partial_b$. 
$ \widetilde Q_b(v)=\|\overline\partial_bv\|^2  +\|\overline\partial_b^\dagger v\|^2$ for $v \in \operatorname{Dom}\overline\partial_b   \cap\operatorname{Dom}\overline\partial_b^\dagger $.
%
Because $M$ is compact,  \(\|\cdot\|\)  and the Euclidean norm \(\|\cdot\|_{\mathrm E}\) are always comparable. 
Let $\chi_1=\sin\frac{\pi\psi}{2}, \chi_2=\cos\frac{\pi\psi}{2}.$ Then
$\chi_1, \chi_2$ are cut-off functions invariant
under rotations in $u$ satisfying 
$ \chi_1^2+\chi_2^2=1$, $\chi_1=1$ near $K$,
and $ \operatorname{supp}\chi_1 \Subset\mathcal F(V_1\times\mathbb S^1),  \operatorname{supp}\chi_2 \subset M\setminus\mathcal F(V_0\times\mathbb S^1).$
In particular, \(\operatorname{supp}\chi_2\) is a compact
subset of the strictly pseudoconvex part of \(M\).

For a real smooth function \(\chi\), it is easy to verify that 
\[
 \overline\partial_b(\chi v)
 =\chi\overline\partial_bv
   +\overline\partial_b\chi\wedge v,
\quad 
 \overline\partial_b^\dagger(\chi v)
 =\chi\overline\partial_b^\dagger v
   -(\overline\partial_b\chi)\mathbin{\lrcorner}_h v,
\]
where $\mathbin{\lrcorner}_h$ denotes the contraction with respect to the Hermitian metric $h$. By
\[
 \sum_{j=1}^2\chi_j\overline\partial_b\chi_j=0, 
\quad
 |\beta\wedge v|^2+|\beta\mathbin{\lrcorner}_h v|^2
 =|\beta|^2|v|^2,
\]
we obtain 
\begin{equation}\label{111}
 \widetilde Q_b(\chi_1v)+\widetilde Q_b(\chi_2v)
 =\widetilde Q_b(v)+\int_M\sum_{j=1}^2
       |\overline\partial_b\chi_j|^2|v|^2\,d\mu,
\end{equation}
where
\begin{align}\label{222}
\int_M\sum_{j=1}^2
       |\overline\partial_b\chi_j|^2|v|^2\,d\mu=\frac{\pi^2}{4}
   \int_M|\overline\partial_b\psi|^2|v|^2\,d\mu \leq C \|v\|^2,
\end{align}
for  some constant $ C>0.$

For \(n\in\mathbb Z\), define the Fourier series
\[
v_n:= \Pi_nv
 :=\frac1{2\pi}\int_0^{2\pi}
       e^{-in\alpha}R_\alpha^*v\,d\alpha.
\]
Then it is easy to verify that $ R_\alpha^*v_n=e^{in\alpha}v_n.$ Moreover, 
Since the metric \(h\) and the density \(d\mu\) are invariant under
\(R_\alpha\), the pullback \(R_\alpha^*\) is unitary on
\(L^2_{(0,q)}(M,h_q,d\mu)\). More precisely,
$|R_\alpha^*v|(p)=|v|(R_\alpha(p))$ and thus
$\|R_\alpha^*v\|=\|v\|.$
Because \(R_\alpha\) is a CR automorphism,
$\overline\partial_bR_\alpha^* =R_\alpha^*\overline\partial_b.$
By taking adjoints and using the unitarity of \(R_\alpha^*\), we also obtain
$\overline\partial_b^\dagger R_\alpha^*=R_\alpha^*\overline\partial_b^\dagger.$
Moreover, 
if \(n\neq m\), it follows that
$$(v_n,v_m)=(R_\alpha^* v_n, R_\alpha^* v_m)=
(e^{in\alpha}v_n,e^{im\alpha}v_m)=e^{i(n-m)\alpha}(v_n,v_m).$$
This implies $(v_n,v_m)=0$ for $n\neq m$.
Therefore, Parseval’s identity yields
$\|v\|^2 = \sum_{n\in\mathbb Z}\|v_n\|^2.$
Also, because \(\overline\partial_b\) and
\(\overline\partial_b^\dagger\) commute with \(R_\alpha^*\), they also commute with $\Pi_n$:
$$ \overline\partial_bv_n= \overline\partial_b\Pi_nv=\Pi_n\overline\partial_bv,\quad  \overline\partial_b^\dagger v_n=\overline\partial_b^\dagger\Pi_nv=\Pi_n\overline\partial_b^\dagger v.$$
Then each $v_n \in  \operatorname{Dom}\overline\partial_b \cap \operatorname{Dom}\overline\partial_b^\dagger$ if $ v \in \operatorname{Dom}\overline\partial_b \cap \operatorname{Dom}\overline\partial_b^\dagger.$
Furthermore, \(\overline\partial_bv_n\) and $\overline\partial_b^\dagger v_n$ have character \(n\), namely,  $R_\alpha^*(\overline\partial_bv_n)=e^{in\alpha}\overline\partial_bv_n$,
 $R_\alpha^*(\overline\partial_b^\dagger v_n)=e^{in\alpha}\overline\partial_b^\dagger v_n.$ Consequently, the families
$\{\overline\partial_bv_n\}_{n\in\mathbb Z}$ 
and $\{\overline\partial_b^\dagger v_n\}_{n\in\mathbb Z}$ are orthogonal. Applying Parseval’s identity 
we obtain
$$
\|\overline\partial_bv\|^2= \sum_{n\in\mathbb Z} \|\overline\partial_bv_n\|^2, \quad  \|\overline\partial_b^\dagger v\|^2=\sum_{n\in\mathbb Z}\|\overline\partial_b^\dagger v_n\|^2.$$
%
It thus follows that $ \widetilde Q_b(v)= \sum_{n\in\mathbb Z} \left( \|\overline\partial_bv_n\|_{h,\mu}^2 + \|\overline\partial_b^\dagger v_n\|_{h,\mu}^2 \right)
=\sum_{n\in\mathbb Z}\widetilde Q_b(v_n)$ for every $v\in \operatorname{Dom}\overline\partial_b \cap \operatorname{Dom}\overline\partial_b^\dagger.$
Since cut-off functions \(\chi_1,\chi_2\) are invariant under rotations of \(u\), it is easy to verify that  
$ \Pi_n(\chi_jv)=\chi_j\Pi_nv$ for $ j=1,2.$

We now compute the quadratic form in \(\mathcal F(V_2\times\mathbb S^1)\). 
For  scalar functions \(f,g\) supported in \(\mathcal F(V_2\times\mathbb S^1)\), it follows from the integration by parts that
\[
 \int_M(\overline L_af)\overline g\,d\mu
 =-\int_M f\,\overline{L_ag}\,d\mu.
\]
Thus, for
\(v=v_1\overline\omega_1+v_2\overline\omega_2\) supported in \(\mathcal F(V_2\times\mathbb S^1)\), we have 
\[
 \overline\partial_bv
 =(\overline L_1v_2-\overline L_2v_1)
   \overline\omega_1\wedge\overline\omega_2,
 \quad
 \overline\partial_b^\dagger v
 =-L_1v_1-L_2v_2.
\]
Define
$ D_n f=\overline\partial f+n\overline\partial\varphi\wedge f$ to be a first order differential operator on the base. 
It follows that 
 the adjoint  $D_n^*$ of $D_n$
with respect to unweighted Euclidean measure on the base is given by 
$ D_n^*g =-\sum_{a=1}^2  \left(\frac{\partial g_a}{\partial z_a}   -n\varphi_{z_a}g_a\right)$ for $g$ supported in $V_2$. 
Therefore, for $ v_n=e^{in\theta}   \sum_{a=1}^2g_a(z,w)\overline\omega_a$ 
 supported in \(\mathcal F(V_2\times\mathbb S^1)\),  the preceding formulas yield 
\[
 \|v_n\|^2=\|g\|_{L^2(B)}^2,
 \quad
 \widetilde Q_b(v_n)
 =\|D_ng\|_{L^2(B)}^2+\|D_n^*g\|_{L^2(B)}^2.
\]
Conversely,  let $f=f_1\,d\overline z+f_2\,d\overline w$
be a base \((0,1)\)-form supported in \(V_1\). Define the boundary
\((0,1)\)-form \(\mathcal U_nf\) by
\[
\bigl(\mathcal U_nf\bigr)
\bigl(\mathcal F(z,w,\theta)\bigr)
=
e^{in\theta}e^{-n\varphi(z,w)}
\left(
f_1(z,w)\overline\omega_1
+
f_2(z,w)\overline\omega_2
\right).
\]
A direct computation yields
$ \|\mathcal U_nf\|^2 = \int_{V_1} \left(|f_1|^2+|f_2|^2\right)e^{-2n\varphi}\,dV = \|f\|_{2n\varphi}^2.$
Moreover, we have 
$ D_n(e^{-n\varphi}f) =e^{-n\varphi}\overline\partial f$
and
$ D_n^*(e^{-n\varphi}f)= e^{-n\varphi} \left[ -\sum_{a=1}^2 \left( \frac{\partial f_a}{\partial z_a}  -2n\varphi_{z_a}f_a \right) \right]
 =e^{-n\varphi}    \overline\partial_{2n\varphi}^*f.$
It follows that
$  \widetilde Q_b(\mathcal U_nf)=Q_{n,\varphi}(f).$

For \(n\geq1\), write $ \chi_1v_n=\mathcal U_nf_n.$
Since \(\operatorname{supp}f_n\Subset V_1\Subset P'\),
\begin{align*}
 \widetilde Q_b(\chi_1v_n)=Q_{n,\varphi}(f_n)\geq\gamma_n(P')\|f_n\|_{2n\varphi}^2=\gamma_n(P')\|\chi_1v_n\|^2.
\end{align*}
On the other hand, define on \(F(V_2\times\mathbb S^1)\), 
\[
 \mathcal J
 \left(v_1\overline\omega_1+v_2\overline\omega_2\right)
 :=
 \overline{v_2}\,\overline\omega_1
 -\overline{v_1}\,\overline\omega_2.
\]
Then by direct calculation, 
$ \|\mathcal Jv\|=\|v\|,  \operatorname{supp}\mathcal Jv=\operatorname{supp}v, R_\alpha^*(\mathcal Jv_n)=e^{-in\alpha}\mathcal Jv_n.$
Moreover,
\begin{align*}
 \overline\partial_b(\mathcal Jv)= -\bigl(\overline L_1\overline{v_1}  +\overline L_2\overline{v_2}\bigr) \overline\omega_1\wedge\overline\omega_2=\overline{\overline\partial_b^\dagger v}\,
       \overline\omega_1\wedge\overline\omega_2,
\end{align*}
and, if $ \overline\partial_bv =F\,\overline\omega_1\wedge\overline\omega_2,$
then
\[
 \overline\partial_b^\dagger(\mathcal Jv)
 =-L_1\overline{v_2}+L_2\overline{v_1}
 =-\overline F.
\]
Consequently,
$ \widetilde Q_b(\mathcal Jv)=\widetilde Q_b(v).$
For \(n\leq-1\), the form \(\mathcal J(\chi_1v_n)\)
has positive character \(-n\), and hence
\begin{align*}
 \widetilde Q_b(\chi_1v_n)
=\widetilde Q_b\bigl(\mathcal J(\chi_1v_n)\bigr) \geq\gamma_{-n}(P')
       \|\mathcal J(\chi_1v_n)\|^2 =\gamma_{|n|}(P')\|\chi_1v_n\|^2.
\end{align*}
Thus, for every \(n\neq0\),
\begin{equation}\label{chi1}
 \widetilde Q_b(\chi_1v_n)
 \geq\gamma_{|n|}(P')\|\chi_1v_n\|^2.
\end{equation}
Let $\mathcal Xv=\left.\frac{d}{d\alpha}R_\alpha^*v\right|_{\alpha=0}$
be the infinitesimal generator of the rotation in $u$ variable. Since
\(R_\alpha^*\) is unitary, $(R_\alpha^*v,R_\alpha^*g)=(v,g).$
Differentiating at \(\alpha=0\) yields
$(\mathcal Xv,g)+(v,\mathcal Xg)=0,$ and hence $\mathcal X^*=-\mathcal X.$
Moreover, if \(v_n\) is the \(n\)-th Fourier component, then 
$R_\alpha^*v_n=e^{in\alpha}v_n.$
Differentiating this identity at \(\alpha=0\) yields
$\mathcal Xv_n=inv_n.$
As a first-order differential operator, \(\mathcal X\) defines a
bounded map
\[
\mathcal X:H^{1/2}(M)\rightarrow H^{-1/2}(M).
\]
Since \(\chi_2\) is invariant under the \(u\)-rotation,
$\mathcal X(\chi_2v_n)=in\chi_2v_n.$
Hence, by duality,
\begin{align*}
|n|\|\chi_2v_n\|^2 = \left| \bigl(\mathcal X(\chi_2v_n),\chi_2v_n\bigr) \right| \leq \|\mathcal X(\chi_2v_n)\|_{H^{-1/2}} \|\chi_2v_n\|_{H^{1/2}} \leq C\|\chi_2v_n\|_{H^{1/2}}^2. 
\end{align*}
Consequently,
\[
(1+n^2)^{1/2}\|\chi_2v_n\|^2
\leq
C\|\chi_2v_n\|_{H^{1/2}}^2.
\]
On the support of \(\chi_2\), which is strictly pseudoconvex, then the 
subelliptic estimate yields
\[
 \|\chi_2v_n\|_{H^{1/2}(M)}^2
 \leq C'
       \bigl(\widetilde Q_b(\chi_2v_n)+\|\chi_2v_n\|^2\bigr).
\]
It follows that
\begin{equation}\label{chi2}
\widetilde Q_b(\chi_2v_n)
\geq
\left(
c (1+n^2)^{1/2}-1
\right)
\|\chi_2v_n\|^2,
\end{equation}
where \(c>0\) is independent of \(n\).
Combining (\ref{111}), (\ref{222}), (\ref{chi1}), (\ref{chi2}), for \(n\neq0\), we have
\begin{align}\label{33}
 \widetilde Q_b(v_n)
 &\geq
 \gamma_{|n|}(P')\|\chi_1v_n\|^2
 +c(1+n^2)^{1/2}\|\chi_2v_n\|^2 -\|\chi_2v_n\|^2-C\|v_n\|^2 \geq\lambda_n\|v_n\|^2,
\end{align}
where
$ \lambda_n:= \min\left\{\gamma_{|n|}(P'),  c(1+n^2)^{1/2}-1 \right\} -C .$
By (\ref{eq:compact-gaps}), $\lambda_n \to +\infty$,
as $|n|\rightarrow\infty$.

We also need an upper bound for each fixed character.
Write $ \chi_1v_n=e^{in\theta} \sum_{a=1}^2g_a(z,w)\overline\omega_a,$
where each \(g_a\) has compact support in \(V_1\).
For the unweighted operators on the base, it follows from the integration by parts that,
\[
 \|\nabla g\|_{L^2}^2
 =4\sum_{a,b=1}^2
       \left\|\frac{\partial g_a}{\partial\overline z_b}
       \right\|_{L^2}^2
 =4\bigl(\|\overline\partial g\|_{L^2}^2
         +\|\overline\partial^*g\|_{L^2}^2\bigr).
\]
It follows from $ \overline\partial g=D_ng-n\overline\partial\varphi\wedge g$ and $ \overline\partial^*g =D_n^*g-n(\overline\partial\varphi) \mathbin{\lrcorner}g$ that
\begin{align*}
 \|\nabla g\|_{L^2}^2 \leq 8\bigl(\|D_ng\|_{L^2}^2+\|D_n^*g\|_{L^2}^2\bigr) +
 8n^2\int_{V_1}|\overline\partial\varphi|^2|g|^2\,dV \leq
 8\widetilde Q_b(\chi_1v_n)
 +C n^2\|\chi_1v_n\|^2.
\end{align*}
Also, note that  $ \|\partial_\theta(\chi_1v_n)\|^2 =n^2\|\chi_1v_n\|^2.$
It then follows from the
equivalence of the Sobolev norms on the fixed
compact set that 
\[
 \|\chi_1v_n\|_{H^1(M)}^2
 \leq C\left(
       \widetilde Q_b(\chi_1v_n)
       +(1+n^2)\|\chi_1v_n\|^2
       \right).
\]
Again by (\ref{111}) and (\ref{222}), 
$ \widetilde Q_b(\chi_jv_n) \leq\widetilde Q_b(v_n)+C\|v_n\|^2$, 
for \(j=1,2\).
Using $ v_n=\chi_1(\chi_1v_n)+\chi_2(\chi_2v_n),$
it follows from the subelliptic estimates on the strictly pseudoconvex set that  
\begin{align}\label{44}
 \|v_n\|_{H^{1/2}(M)}^2
 &\leq C\left(
       \|\chi_1v_n\|_{H^1(M)}^2
       +\|\chi_2v_n\|_{H^{1/2}(M)}^2
       \right) \leq C\left(
       \widetilde Q_b(v_n)+(1+n^2)\|v_n\|^2
       \right).
\end{align}

Applying the density lemma, all these estimates extend to \(\operatorname{Dom}\overline\partial_b   \cap\operatorname{Dom}\overline\partial_b^\dagger \) . The density lemma states that, for every \(v\in \operatorname{Dom}\overline\partial_b   \cap\operatorname{Dom}\overline\partial_b^\dagger \), there are smooth forms \(v^{(k)}\in C^\infty_{(0,1)}(M) \cap \operatorname{Dom}\overline\partial_b   \cap\operatorname{Dom}\overline\partial_b^\dagger \), such that $\|v^{(k)}-v\|+\|\overline\partial_b(v^{(k)}-v)\|+\|\overline\partial_b^\dagger(v^{(k)}-v)\| \rightarrow 0.$
Thus \(\widetilde Q_b(v^{(k)}-v)\rightarrow0, \widetilde Q_b(v^{(k)})\rightarrow\widetilde Q_b(v).\)
Because $\overline\partial_b\Pi_n= \Pi_n\overline\partial_b, \overline\partial_b^\dagger\Pi_n=\Pi_n\overline\partial_b^\dagger$, 
$ \|\Pi_n(v^{(k)}-v)\|^2 + \|\overline\partial_b\Pi_n(v^{(k)}-v)\|^2 + \|\overline\partial_b^\dagger\Pi_n(v^{(k)}-v)\|^2 \rightarrow0.$ Also
$ \|\chi\Pi_n(v^{(k)}-v)\|^2 + \|\overline\partial_b\left(\chi\Pi_n(v^{(k)}-v)\right)\|^2 + \|\overline\partial_b^\dagger\left(\chi\Pi_n(v^{(k)}-v)\|^2 \right) \rightarrow0$ for a smooth cut-off function $\chi$. 

Set $ S_N:=\sum_{|n|\leq N}\Pi_n,  \Lambda_N:=\inf_{|n|>N}\lambda_n.$
For sufficiently large \(N\),
$ \Lambda_N\rightarrow +\infty$,
and
\begin{align}\label{eq:compact-high-frequency}
 \|(I-S_N)v\|^2 =\sum_{|n|>N}\|v_n\|^2 \leq\frac1{\Lambda_N}  \sum_{|n|>N}\widetilde Q_b(v_n) \leq\frac1{\Lambda_N}\widetilde Q_b(v).
\end{align}
For each fixed \(N\),
\begin{align}\label{eq:compact-low-frequency}
 \|S_Nv\|_{H^{1/2}(M)}^2 \leq(2N+1)\sum_{|n|\leq N}  \|v_n\|_{H^{1/2}(M)}^2 \leq C_N\bigl(\widetilde Q_b(v)+\|v\|^2\bigr).
\end{align}
Moreover, by  Sobolev interpolation inequality that, for every
\(\delta>0\),
\begin{equation}\label{55}
\|f\|^2
\leq
\delta\|f\|_{H^{1/2}(M)}^2
+
C_\delta\|f\|_{H^{-1}(M)}^2.
\end{equation}
Since the rotations \(R_\alpha\) form a compact smooth family of diffeomorphisms of the compact manifold \(M\). Consequently, 
$ \|R_\alpha^*v\|_{H^{-1}(M)} \leq C\|v\|_{H^{-1}(M)}$
holds for 
where \(C>0\) is independent of $0\leq\alpha\leq2\pi$. Therefore,
\[
\begin{aligned}
\|\Pi_nv\|_{H^{-1}(M)} \leq
\frac1{2\pi}
\int_0^{2\pi}
\left|e^{-in\alpha}\right|
\|R_\alpha^*v\|_{H^{-1}(M)}\,d\alpha
\leq
C\|v\|_{H^{-1}(M)}.
\end{aligned}
\]
It follows that 
\begin{align}\label{66}
\|S_Nv\|_{H^{-1}(M)}
\leq
\sum_{|n|\leq N}\|\Pi_nv\|_{H^{-1}(M)}
\leq
(2N+1)C\|v\|_{H^{-1}(M)}.
\end{align}
Applying (\ref{55}) to \(f=S_Nv\) and using \eqref{eq:compact-low-frequency}, (\ref{66}), 
we obtain
\begin{align}
\|S_Nv\|^2
\leq
\delta\|S_Nv\|_{H^{1/2}(M)}^2 +
C_\delta\|S_Nv\|_{H^{-1}(M)}^2  \leq
\delta C_N
\bigl(\widetilde Q_b(v)+\|v\|^2\bigr)
+
C_{\delta,N}\|v\|_{H^{-1}(M)}^2.
\label{eq:compact-low-interpolation}
\end{align}
Then 
\eqref{eq:compact-high-frequency} and
\eqref{eq:compact-low-interpolation} imply
\begin{align}
\|v\|^2 =
\|S_Nv\|^2+\|(I-S_N)v\|^2 \notag \leq
\left(
\delta C_N+\frac{1}{\Lambda_N}
\right)\widetilde Q_b(v)
+
\delta C_N\|v\|^2
+
C_{\delta,N}\|v\|_{H^{-1}(M)}^2.
\label{eq:compact-pre-estimate}
\end{align}
Let \(\varepsilon>0\). First choose \(N\) sufficiently large that
$\frac{1}{\Lambda_N}\leq\frac{\varepsilon}{4}.$
With this \(N\) fixed, choose \(\delta>0\) sufficiently small that
$ \delta C_N \leq \min\left\{\frac12,\frac{\varepsilon}{4}\right\},$
 we obtain
\begin{equation}
\|v\|^2 \leq \varepsilon\widetilde Q_b(v) + C_\varepsilon\|v\|_{H^{-1}(M)}^2,
\end{equation}
for $v \in C^\infty_{(0,1)}(M) \operatorname{Dom}\overline\partial_b \cap \operatorname{Dom}\overline\partial_b^\dagger$. 
By the density lemma, 
 we have shown that the complex Green operator  with respect to $h, d\mu$ is compact.

It follows that  the complex Green operator $G_1$ is also compactness.  
The relevant  argument for compactness independent of the Hermitian metric
is given in Section 2 of \cite{CS09} for the $\bar\partial$-Neumann
operator. The argument for
 the complex Green operator is similar, using
the characterization of compactness of $G_1$ by compactness of
the canonical $\bar\partial_b$ solution operators in degrees
$1$ and $2$ (cf. 
Section 4, p. 773, equation (25) and the following paragraph in \cite{RS08}).
The reason this argument carries over is that $\bar\partial_b$
and its solution spaces are intrinsic to the CR structure.
Changing the inner product changes the canonical solution only
through the orthogonal projection that selects the solution
orthogonal to $\ker\bar\partial_b$. These projections are bounded,
and composition with bounded operators preserves compactness.
Moreover, smooth positive metrics and densities on compact $M$
product equivalent $L^2$ norms, so boundedness and relative compactness
are unchanged. A change of density function is also acceptable by absorbing it into the Hermitian inner product. 
\end{proof}

Now compactness of $N_1$ follows from the result of Raich and Straube (cf. Theorem 1.1 in \cite{RS08}).

\begin{proposition}\label{cor:compact-N1}
The \(\overline\partial\)-Neumann operator \(N_1\) on \(\Omega\) is
compact.
\end{proposition}

\noindent Qianyun Wang, wangqy1226@gmail.com, School of Statistics and Mathematics, Shanghai Lixin University of Accounting and Finance, Shanghai 201209, China \\

\noindent Yuan Yuan, yuanyuan@westlake.edu.cn, Institute for Theoretical Sciences, Westlake University,\\
Hangzhou 310024, Zhejiang, China\\
  
\noindent Xu Zhang, xzhangmath@tongji.edu.cn, School of Mathematical Sciences, Tongji University, Shanghai 200092, China.

\end{document}